\documentclass[12pt,english,reqno, a4paper]{amsart}
\usepackage[T1]{fontenc}
\usepackage[latin9]{inputenc}
\usepackage[a4paper]{geometry}
\usepackage{textcomp}
\usepackage{mathrsfs}
\usepackage{amstext}
\usepackage{amsthm}
\usepackage{amssymb}

\numberwithin{equation}{section}
\numberwithin{figure}{section}
\theoremstyle{plain}
\newtheorem{thm}{\protect\theoremname}[section]

  \theoremstyle{definition}

  \newtheorem{rem}[thm]{Remark}
  \theoremstyle{plain}
  \newtheorem{lem}[thm]{\protect\lemmaname}
\newtheorem{corollary}[thm]{Corollary}
\usepackage{mathrsfs}
\usepackage{color}
\usepackage{hyperref}
\hypersetup{
   colorlinks = true,%
   citecolor = blue,
   filecolor=red,%
   linkcolor = [rgb]{0.65,0.0,0.0},%
   anchorcolor = red,
   pagecolor = red,
   urlcolor= [rgb]{0.65,0.0,0.0}
   linktocpage=true,
   pdfpagelabels=true,
   bookmarksnumbered=true,
}
\newtheorem{innercustomthm}{\rm \textbf{Theorem}}
\newenvironment{customthm}[1]
{\renewcommand\theinnercustomthm{#1}\innercustomthm}
{\endinnercustomthm}
\newcommand{\im}[1]{\displaystyle\int_{M}{#1}\,{\rm d}v_g}
\def\d{{{\rm d}}v_{g}}
\def\E{\mathcal{E}}
\newcommand{\R}{{\mathbb R}}
\newcommand{\ds}{\displaystyle}
\newcommand{\Pl}{(\mathscr{P}_\lambda)}

\makeatother

\usepackage{babel}
  \providecommand{\examplename}{Example}
  \providecommand{\lemmaname}{Lemma}
  
\providecommand{\theoremname}{Theorem}

\begin{document}

\title[A characterization of Schrödinger equations ]{A characterization  related to Schrödinger equations on
Riemannian manifolds}

\author{Francesca Faraci }

\email{ffaraci@dmi.unict.it}

\address{Department of Mathematics and Computer Science, University of Catania,
Catania, Italy}

\author{Csaba Farkas}

\email{farkas.csaba2008@gmail.com \& farkascs@ms.sapientia.ro}

\address{Department of Mathematics and Computer Science, Sapientia University,
Tg. Mures, Romania \& Institute of Applied Mathematics, Óbuda University,
1034 Budapest, Hungary}

\date{6 April 2017}
\begin{abstract}
In this paper we consider the following problem
\[
\left\{ \begin{array}{ll}
-\Delta_{g}u+V(x)u=\lambda\alpha(x)f(u), & \mbox{in }M\\
u\geq 0, & \mbox{in }M\\
u\to0, & \mbox{as }d_{g}(x_{0},x)\to\infty
\end{array}\right.\eqno{(\mathscr{P}_{\lambda})}
\]where $(M,g)$ is a $N$-dimensional ($N\geq3)$, non-compact Riemannian manifold with asymptotically non-negative Ricci curvature, $\lambda$ is a real parameter, $V$ is a positive coercive potential, $\alpha$ is a bounded function and $f$ is a suitable nonlinearity. By using variational methods we prove a characterization result for existence of solutions for $(\mathscr{P}_{\lambda})$.
\end{abstract}


\maketitle

\section{Introduction}

The existence of standing waves solutions for the nonlinear Schrödinger
equation
\[
i\hbar\frac{\partial\psi}{\partial t}=-\frac{\hbar^{2}}{2m}\Delta\psi+V(x)\psi-f(x,|\psi|),\ \ \ \text{\ensuremath{\mathrm{in}} }\mathbb{R}^{N}\times\mathbb{R}_{+}\setminus\{0\},
\]
has been intensively studied in the last decades.
The Schrödinger equation  plays a central role  in quantum mechanic as it predicts the future behavior of a dynamic system. Indeed, the wave function $\psi(x,t)$ represents the quantum mechanical
probability amplitude for a given unit-mass particle to have position
$x$ at time $t$. Such equation appears in  several fields of physics, from  Bose\textendash Einstein condensates
and nonlinear optics, to plasma physics (see for instance  \cite{ARMA2002, TAMS2008} and reference therein).

A Lyapunov-Schmidt type reduction, i.e. a separation of variables of the type $\psi(x,t)=u(x)e^{-i\frac{E}{\hbar}t}$, leads to the following semilinear
elliptic equation
\begin{equation*}
-\Delta u+V(x)u=f(x,u),\ \ \ \mbox{ in }\mathbb{R}^{N}.\label{eq:general}
\end{equation*}

With the aid of variational methods, the existence and multiplicity of nontrivial solutions for such problems have been extensively studied in the literature over the last decades. For instance, the existence of positive solutions  when the potential $V$ is coercive and $f$ satisfies standard mountain pass assumptions,  are well known after the seminal paper of Rabinowitz \cite{Rabinowitz92ZAMP}. Moreover, in the class of bounded from below potentials,  several attempts have been made to  find general assumptions  on $V$ in order to obtain existence and multiplicity
results (see for instance \cite{BartschPankoWang01CCM,BartschWang95CPDE,Benci-Fortunato-JMAA,Willem-book,Strauss}). In such papers the nonlinearity $f$ is required to satisfy the well-know Ambrosetti-Rabinowitz  condition, thus it is  superlinear
at infinity. For a sublinear growth  of $f$ see also  \cite{Kristaly-NODEA}.

Most of the  aforementioned papers provide \textit{sufficient} conditions on the nonlinear term $f$ in order to prove  existence/multiplicity type results.
The novelty of the present paper is to establish a \textit{characterization} result for stationary Schrödinger equations on unbounded domains; even more, our arguments work on not necessarily linear structures. Indeed, our results fit the research direction where the solutions of certain PDEs are influenced by the geometry of the ambient structure (see for instance \cite{FKVCalculus,FK16,Kristaly-DCDS2009,Kristaly12,LiYau,Ma} and reference therein). Accordingly, we deal with a Riemannian setting, the results on $\mathbb{R}^N$ being a particular consequence of our general achievements.



In order to give the precise statement of our result, let us denote by $(M,g)$  a $N$-dimensional ($N\geq3)$, complete, non-compact Riemannian manifold
with asymptotically non-negative Ricci curvature with a base point
$\tilde{x}_{0}\in M$, i.e.,
\begin{itemize}
\item[\textbf{(C)}]  ${\rm Ric}_{(M,g)}(x)\geq-(N-1)H(d_{g}(\tilde{x}_{0},x))$, for
all $x\in M,$ where $H\in C^{1}([0,\infty))$ is a non-negative bounded
function satisfying ${\displaystyle \int_{0}^{\infty}tH(t){\rm d}t=b_{0}<+\infty},$
\end{itemize}
(here and in the sequel $d_g$ is the distance function associated to the Riemannian metric $g$). For an overview on such property see \cite{Adriano-Xia2010, Pigola-Rigoli-Setti2008}.

Let  $x_{0}\in M$
be a fixed point,  $\alpha:M\to\mathbb{R}_{+}\setminus\{0\}$ a bounded function and  $f:\mathbb{R}_{+}\to\mathbb{R}_{+}$  a continuous
function with $f(0)=0$ such that there exist two constants $C>0$
and $q\in(1,2^{\star})$ (being $2^\star$ the Sobolev critical exponent) such that
\begin{equation}
f(\xi)\leq k\left(1+\xi^{q-1}\right)\ \mbox{for all}\ \xi\geq0.\label{growth2}
\end{equation}
Denote by $F:\mathbb{R}_{+}\to\mathbb{R}_{+}$ the function $\ds F(\xi)=\int_0^\xi f(t)dt$.

We assume that $V:M\to\mathbb{R}$ is a measurable function satisfying the following conditions:
\begin{itemize}
	\item[$(V_{1})$] ${\displaystyle V_0={\rm essinf}_{x\in M}V(x)>0}$;
	\item[$(V_{2})$] ${\displaystyle \lim_{d_{g}(x_{0},x)\to\infty}V(x)=+\infty,}$ for
	some $x_{0}\in M$.
\end{itemize}

The problem we deal with is written as:

\[
\left\{ \begin{array}{ll}
-\Delta_{g}u+V(x)u=\lambda\alpha(x)f(u), & \mbox{in }M\\
u\geq 0, & \mbox{in }M\\
u\to0, & \mbox{as }d_{g}(x_{0},x)\to\infty.
\end{array}\right.\eqno{(\mathscr{P}_{\lambda})}
\]

Our result reads as follows:
\begin{thm}
\label{thm:main} Let $N\geq3$ and $(M,g)$ be a complete, non-compact $N-$dimensional Riemannian manifold satisfying the curvature condition $(\mathbf{C})$, and ${\displaystyle \inf_{x\in M}\mathrm{Vol}_{g}(B_{x}(1))>0}.$
Let also $\alpha:M\to\mathbb{R}_{+}\setminus\{0\}$ be in $L^{\infty}(M)\cap L^1(M)$,
$f:\mathbb{R}_{+}\to\mathbb{R}_{+}$  a continuous function with
$f(0)=0$ verifying (\ref{growth2}) and $V:M\to\mathbb{R}$ be a
potential verifying $(V_{1})$, $(V_{2})$. Assume that
for some $a>0$, the function ${\displaystyle \xi\to\frac{F(\xi)}{\xi^{2}}}$ is  non-increasing
 in $(0,a]$. Then, the following conditions are equivalent:
\begin{itemize}
\item[$(i)$] for each $b>0$, the function ${\displaystyle \xi\to\frac{F(\xi)}{\xi^{2}}}$ is
not constant in $(0,b]$;
\item[$(ii)$] for each $r>0$, there exists an open interval $I_r\subseteq(0,+\infty)$
such that for every $\lambda\in I_r$, problem $(\mathcal{\mathscr{P}}_{\lambda})$
has a nontrivial solution $u_{\lambda}\in H_{g}^{1}(M)$ satisfying
\[
{\displaystyle \int_{M}\left(|\nabla u_{\lambda}(x)|^{2}+V(x)u_{\lambda}^{2}\right)dv_{g}<r}.
\]
\end{itemize}
\end{thm}
\begin{rem}
\begin{itemize}
	\item[(a)] One can replace the assumption ${\displaystyle \inf_{x\in M}\mathrm{Vol}_{g}(B_{x}(1))>0}$ with a curvature restriction, requiring that the sectional curvature is bounded from above. Indeed, using the Bishop-Gromov theorem one can easily get that ${\displaystyle \inf_{x\in M}\mathrm{Vol}_{g}(B_{x}(1))>0}$.
	\item[(b)] A more familiar form of Theorem \ref{thm:main} can be obtained when $Ric_{(M,g)}\geq 0$; it suffices to put $H\equiv 0$ in \textbf{(C)}.
\end{itemize}
\end{rem}

\noindent The following potentials  $V$ fulfills  assumptions $(V_1)$ and $(V_2)$:
	\begin{enumerate}
		\item[(i)]  Let $V(x)=d_{g}^{\theta}(x,x_{0})+1$, where $x_{0}\in M$ and $\theta>0$.
		\item[(ii)] More generally, if  $z:[0,+\infty)\to[0,+\infty)$ is a bijective function, with $z(0)=0$, let  $V(x)=z(d_{g}(x,x_{0}))+c,$
		where $x_{0}\in M$ and $c>0.$
	\end{enumerate}
The work is motivated by a result of Ricceri (\cite{Ricerri_Characterization}) where a similar theorem is stated for one-dimensional Dirichlet problem; more precisely, $(i)$ from Theorem \ref{thm:main} characterizes the existence of the solutions for the following problem  \[
		\left\{ \begin{array}{ll}
		-u''=\lambda\alpha(x)f(u), & \mbox{in }(0,1)\\
		u>0, & \mbox{in }(0,1)\\
		u(1)=u(0)=0.
		\end{array}\right.
		\]
		In the above theorem it is crucial the embedding of the Sobolev space $H^1_0((0,1))$ into $C^0([0,1])$.

Recently, this result has been extended by Anello to higher dimension,
i.e. when the interval $(0,1)$ is replaced by a bounded domain $\Omega\subset\mathbb{R}^{N}$ ($N\in \mathbb N$) with smooth boundary (\cite{Anello}). The generalization
follows by  direct minimization procedures and contains a more precise information on the interval of parameters $I$. See also \cite{BisciRadulescu} for a similar characterization in the framework of fractal sets.

Let us note that in our setting the situation is much more delicate with respect to those treated in the papers \cite{Anello,Ricerri_Characterization}. Indeed, the Riemannian framework produces several technical difficulties that we overcome by
using an appropriate variational formulation.

One of the main tools in our investigation is a recent result by Ricceri
\cite{Ricceri-2014} (see Theorem \ref{D}
in Section \ref{preliminaries}). The main difficulty in the implication
$(i)\Rightarrow(ii)$ in Theorem \ref{thm:main}, consists in proving  the
boundedness of the solutions. To overcome this difficulty we use the
Nash-Moser iteration method adapted to the Riemannian setting.

In proving $(ii)\Rightarrow(i)$, we make use of a recent result by Poupaud \cite{Poupaud2005} (see Theorem \ref{C} in Section \ref{preliminaries}) concerning the discreteness of the
spectrum of the operator $u\mapsto-\Delta_{g}u+V(x)u$. It is worth mentioning that such result was first obtained by  Kondrat\textquoteright ev and Shubin (\cite{Kondratev*Shubin})  for manifolds with bounded geometry and relies on the generalization of Molchanov\textquoteright s
criterion. However, since the bounded geometry property is a strong assumption
and implies the positivity of the radius of injectivity, many efforts have been made for improvement and generalizations. Later, Shen
\cite{Shen-PAMs-2003} characterized the discretness of the spectrum
by using the basic length scale function and the effective potential
function. For further recent studies in this topic, we invite the
reader to consult the papers \cite{CianchiMazya-JDG,CianchiMazya-AMJ,CAG-2016}.

The outline of the paper is as follows. In §2 we present a series
of preparatory definitions and results which are used throughout the
paper. In §3 we prove our main result.

\section{Preliminaries}\label{preliminaries}

\subsection{Elements from Riemannian geometry}

In the sequel, let $N\geq3$ and $(M,g)$ be an $N-$dimens\-ional Riemannian
manifold. Set also $T_{x}M$ its tangent space at $x\in M$, ${\displaystyle TM=\bigcup_{x\in M}T_{x}M}$
 the tangent bundle, and $d_{g}:M\times M\to[0,+\infty)$  the
distance function associated to the Riemannian metric $g$. Let $B_{x}(\rho)=\{y\in M:d_{g}(x,y)<\rho\}$
be the open metric ball with center $x$ and radius $\rho>0$. If
${\text{d}}v_{g}$ is the canonical volume element on $(M,g)$, the
volume of an open bounded set $\Omega\subset M$ is $\mathrm{Vol}{}_{g}(\Omega)=\int_{\Omega}{\text{d}}v_{g}=\mathcal{H}^{N}(\Omega)$,
where $\mathcal{H}^{N}(S)$ denotes the $N-$dimensional Hausdorff
measure of $\Omega$ with respect to the metric $d_{g}$. The manifold
$(M,g)$ has Ricci curvature bounded from below if there exists $h\in\mathbb{R}$
such that ${\rm Ric}_{(M,g)}\geq hg$ in the sense of bilinear forms,
i.e., ${\rm Ric}_{(M,g)}(X,X)\geq h|X|_{x}^{2}$ for every $X\in T_{x}M$
and $x\in M,$ where ${\rm Ric}_{(M,g)}$ is the Ricci curvature,
and $|X|_{x}$ denotes the norm of $X$ with respect to the metric
$g$ at the point $x$.
The behavior of the volume of  geodesic
balls is given by the following theorem (see \cite{GallotHulinLafontain,Pigola-Rigoli-Setti2008}):

\begin{customthm}{\textbf{A}}\label{A}\cite[Corollary 2.17]{Pigola-Rigoli-Setti2008}, \cite{Adriano-Xia2010}. Let $(M,g)$ be an $N-$dimensional
complete Riemannian manifold. 
If $(M,g)$ satisfies the curvature condition $\textbf{(C)}$,  then
the following volume growth property holds true:
\[
\frac{\mathrm{Vol}_{g}(B_{x}(R))}{\mathrm{Vol}_{g}(B_{x}(r))}\leq e^{(N-1)b_{0}}\left(\frac{R}{r}\right)^{N},\ {0}<r<R,
\]
and  $${\displaystyle\mathrm{Vol}_{g}(B_{x}(\rho))\leq e^{(N-1)b_{0}}\omega_{N}}\rho^{N}, \ \rho>0.$$

where $b_0$ is from condition $({\bf C})$.
\end{customthm}

\noindent Let $p>1.$ The norm of $L^{p}(M)$ is given by
\[
\|u\|_{L^{p}(M)}=\left({\displaystyle \int_{M}|u|^{p}\mathrm{d}v_{g}}\right)^{1/p}.
\]
Let $u:M\to\mathbb{R}$ be a function of class $C^{1}.$ If $(x^{i})$
denotes the local coordinate system on a coordinate neighbourhood
of $x\in M$, and the local components of the differential of $u$
are denoted by $u_{i}=\frac{\partial u}{\partial x_{i}}$, then the
local components of the gradient $\nabla_{g}u$ are $u^{i}=g^{ij}u_{j}$.
Here, $g^{ij}$ are the local components of $g^{-1}=(g_{ij})^{-1}$.
In particular, for every $x_{0}\in M$ one has the eikonal equation
\[
|\nabla_{g}d_{g}(x_{0},\cdot)|=1\ {\rm \ on}\ M\setminus\{x_{0}\}.
\]
The Laplace-Beltrami operator is given by $\Delta_{g}u={\rm div}(\nabla_{g}u)$
whose expression in a local chart of associated coordinates $(x^{i})$
is
\[
\Delta_{g}u=g^{ij}\left(\frac{\partial^{2}u}{\partial x_{i}\partial x_{j}}-\Gamma_{ij}^{k}\frac{\partial u}{\partial x_{k}}\right),
\]
where $\Gamma_{ij}^{k}$ are the coefficients of the Levi-Civita connection.
The $L^{p}(M)$ norm of $\nabla_{g}u(x)\in T_{x}M$ is given by
\[
\|\nabla_{g}u\|_{L^{p}(M)}=\left({\displaystyle \int_{M}|\nabla_{g}u|^{p}{\rm d}v_{g}}\right)^{\frac{1}{p}}.
\]
The space $H_{g}^{1}(M)$ is the completion of $C_{0}^{\infty}(M)$
with respect to the norm
\[
\|u\|_{H_{g}^{1}(M)}=\sqrt{\|u\|_{L^{2}(M)}^{2}+\|\nabla_{g}u\|_{L^{2}(M)}^{2}}.
\]

\subsection{Variational  tools}

\bigskip

Let us consider the functional space
\[
H_{V}^{1}(M)=\left\{ u\in H_{g}^{1}(M):\im{\left(|\nabla_{g}u|^{2}+V(x)u^{2}\right)}<+\infty\right\}
\]
endowed with the norm
\[
\|u\|_{V}=\left(\im{|\nabla_{g}u|^{2}}+\im{V(x)u^{2}}\right)^{1/2}.
\]

It was proved by Aubin \cite{Aubin}
and independently by Cantor \cite{Cantor} that the Sobolev embedding $H_{g}^{1}(M)\hookrightarrow L^{2^{*}}(M)$ is continuous
for complete manifolds with bounded sectional curvature
and positive injectivity radius. The above result was generalized
(\cite{Hebey-BOOK}) for manifolds with  Ricci curvature bounded from below
and positive injectivity radius.
Taking into account that, if $(M,g)$ is an $N$\textminus dimensional complete non-compact
Riemannian manifold with Ricci curvature bounded from below and positive
injectivity radius, then ${\displaystyle \inf_{x\in M}\mathrm{Vol}_{g}(B_{x}(1))>0}$ (\cite{Croke}), we have the following result:
\begin{customthm}{{\textbf{B}}}\textit{\label{B}\cite{Hebey-BOOK, Varaopoulos} Let $(M,g)$ be
		a  complete, non-compact $N$-dimensional Riemannian manifold such that
		its Ricci curvature is bounded from below and ${\displaystyle \inf_{x\in M}\mathrm{Vol}_{g}(B_{x}(1))>0}.$
		Then the embedding $H_{g}^{1}(M)\hookrightarrow L^{p}(M)$ is continuous
		for $p\in[2,2^{*}].$ }\end{customthm}

It is clear that if $(M,g)$ is a Riemannian manifold satisfying the curvature condition \textbf{(C), }and ${\displaystyle \inf_{x\in M}\mathrm{Vol}_{g}(B_{x}(1))>0}$ then the above theorem holds true.
If $V$ is bounded from below by a positive constant, it is clear that the embedding $H_{V}^{1}(M)\hookrightarrow H_{g}^{1}(M)$ is continuous and thus, the above result is still true replacing $H_{g}^{1}(M)$ with $H_{V}^{1}(M)$.

In order to employ  a variational approach we need the next Rabinowitz-type compactness result (see
Rabinowitz \cite{Rabinowitz92ZAMP}):
\begin{lem}
	Let  $(M,g)$ be a complete, non-compact $N-$dimensional Riemannian manifold
	satisfying the curvature condition $\textbf{(C)}$, and ${\displaystyle \inf_{x\in M}\mathrm{Vol}_{g}(B_{x}(1))>0}.$
	If $V$ satisfies $(V_{1})$ and $(V_{2})$, the embedding $H_{V}^{1}(M)\hookrightarrow L^{p}(M)$
	is compact for all $p\in [2,2^{*})$.
\end{lem}
\begin{proof}
	Let $\{u_{k}\}_{k}\subset H_{V}^{1}(M)$ be a bounded sequence, i.e., $\|u_{k}\|_{V}\leq\eta$ for some $\eta>0.$ Since $H_{V}^{1}(M)\hookrightarrow H_{g}^{1}(M)$ is continuous and $H_{g}^{1}(M)\hookrightarrow L_{{\rm loc}}^{2}(M)$ is compact, we can find  $u\in H_{V}^{1}(M)$ such that $u_k\rightharpoonup u$ in $H^1_g(M)$ and $u_k\rightarrow u$ in $L^2_{\rm loc}(M)$ (up to a subsequence).
	Let $\varepsilon>0$ and choose  $q=q(\varepsilon)>0$ big enough. By $(V_{2})$, there exists $R>0$
	such that $V(x)\geq q$ for every $x\in M\setminus B_{R}(x_{0})$.
	Thus,
	\[
	\int_{M\setminus B_{R}(x_{0})}|u_{k}-u|^{2}\d\leq\frac{1}{q}\int_{M\setminus B_{R}(x_{0})}V(x)|u_{k}-u|^{2}\d\leq\frac{(\eta+\|u\|_{V})^{2}}{q}<\frac{\varepsilon}{2}.
	\]
	On the other hand, for $k$ big enough
	\[
	\int_{ B_{R}(x_{0})}|u_{k}-u|^{2}\d<\frac{\varepsilon}{2},
	\]
	and we deduce at once  that $u_{k}\to u$ in $L^{2}(M)$.
	Now, let $p\in(2, 2^*)$ and $\theta=\frac{N}{p}\left(1-\frac{p}{2^{*}}\right) $. It is
	clear that $\theta\in(0,1).$ Then, using  Hölder inequality one
	can see that for $u\in H^1_g(M)$,
	\[
	\im{|u|^{p}}=\im{|u|^{\theta p}\cdot|u|^{(1-\theta)p}}\leq\left(\im{\left(|u|^{\theta p}\right)^{\frac{2}{\theta p}}}\right)^{\frac{\theta p}{2}}\cdot\left(\im{\left(|u|^{(1-\theta)p}\right)^{\frac{2^{*}}{(1-\theta)p}}}\right)^{(1-\theta)\frac{p}{2^{*}}},
	\]
	or
	\[
	\|u\|_{L^{p}(M)}\leq\|u\|_{L^{2}(M)}^{\theta}\cdot\|u\|_{L^{2^{*}}(M)}^{1-\theta}.
	\]
	Thus,
	\[
	 \|u_{k}-u\|_{L^{p}(M)}\leq\|u_{k}-u\|_{L^{2^{*}}(M)}^{1-\theta}\|u_{k}-u\|_{L^{2}(M)}^{\theta}\leq\mathcal{C}\|\nabla_{g}(u_{k}-u)\|_{L^{2}(M)}^{1-\theta}\|u_{k}-u\|_{L^{2}(M)}^{\theta},
	\]
being	$\mathcal{C}>0$ the embedding constant of $H^1_g(M)\hookrightarrow L^{2^*}(M)$  . Therefore,
	$u_{k}\to u$ in $L^{p}(M)$.
\end{proof}

To prove our main results we use the following abstract result due to
Ricceri (the same exploited in \cite{Ricerri_Characterization} for
the study of the one-dimensional case):

\begin{customthm}{\textbf{C}}\textit{\label{D}\cite[Theorem A] {Ricceri-2014} Let $(X,\langle\cdot,\cdot\rangle)$
be a real Hilbert space, $J:X\to\mathbb{R}$ a sequentially weakly
upper semicontinuous and Gâteaux differentiable functional, with $J(0)=0$.
Assume that, for some $r>0$, there exists a global maximum $\hat{x}$
of the restriction of $J$ to $B_{r}=\{x\in X:\|x\|^{2}\leq r\}$
such that
\begin{equation}
J'(\hat{x})(\hat{x})<2J(\hat{x}).\label{condition 1}
\end{equation}
Then, there exists an open interval $I\subseteq(0,+\infty)$ such
that, for each $\lambda\in I$, the equation $x=\lambda J'(x)$ has
a non-zero solution with norm less than $r$.} \end{customthm}

As it was already pointed out in \cite{Ricerri_Characterization}, the
following remark adds some crucial information about the interval
$I$:
\begin{rem}
\label{I}Set
$\ds
\beta_{r}=\sup_{B_{r}}J,\ \ \delta_{r}=\sup_{x\in B_{r}\setminus\{0\}}\frac{J(x)}{\|x\|^{2}}
$
and
$\ds
\eta(s)=\sup_{y\in B_{r}}\frac{r-\|y\|^{2}}{s-J(y)},\ \ \mbox{for all}\ s\in\left(\beta_{r},+\infty\right).
$
Then, $\eta$ is convex and decreasing in $]\beta_{r},+\infty[$.
Moreover,
$\ds
I=\frac{1}{2}\eta(\left(\beta_{r},r\delta_{r}\right)).
$
\end{rem}

\subsection{On the spectrum of $-\Delta_{g}+V(x)$}

In this subsection we recall a key tool  on  the discreteness of the spectrum of the
operator $u\mapsto-\Delta_{g}u+V(x)u$ which we state in a conveniente form for our purposes:

\begin{customthm}{{\textbf{D}}}\label{C}\cite[Corollary 0.1]{Poupaud2005}\textit{ Let $(M,g)$ be a
		complete, non-compact $N$-dimensional Riemannian manifold.
		Let $V:M\to\mathbb{R}$ be a potential verifying $(V_{1})$,
		$(V_{2})$. Assume the following on the manifold $M$:}
	\begin{itemize}
		\item[($A_{1}$)] \textit{there exists $r_{0}>0$ and $C_{1}>0$ such that for any
			$0<r\leq\frac{r_{0}}{2}$, one has $\mathrm{Vol}_{g}(B_{x}(2r))\leq C_{1}\mathrm{Vol}_{g}(B_{x}(r))$ (doubling property);}
		\item[($A_{2}$)] \textit{there exists $q>2$ and $C_{2}>0$ such that for all balls $B_{x}(r)$,
			with $r\leq\frac{r_{0}}{2}$ and for all $u\in H_{g}^{1}(B_{x}(r))$
			\[
			\left(\int_{B_{x}(r)}\left|u-u_{B_{x}(r)}\right|^{q}\mathrm{d}v_{g}\right)^{\frac{1}{q}}\leq C_{2}r\mathrm{Vol}_{g}(B_{x}(r))^{\frac{1}{q}-\frac{1}{2}}\left(\int_{B_{x}(r)}|\nabla_{g}u|^{2}\mathrm{d}v_{g}\right)^{\frac{1}{2}},
			\]
			where $u_{B_{x}(r)}={\displaystyle \frac{1}{\mathrm{Vol}_{g}(B_{x}(r))}\int_{B_{x}(r)}u\mathrm{d}v_{g}}$\qquad(Sobolev- Poincaré inequality).}
	\end{itemize}
	 Then the spectrum of the operator  $-\Delta_{g}+V(x)$
		is discrete.
	
\end{customthm}

It is clear that in our setting condition $(A_1)$ holds (see Theorem \ref{A}).
It was proved by  Maheux
and Saloff-Coste (see for instance \cite{Maheux-Saloff-Coste,Hajlasz-Koskela-2}) that  the Sobolev- Poincaré inequality is true
for complete non-compact Riemannian
manifolds with Ricci curvature bounded from below, thus  Theorem \ref{C} is valid
 for Riemannian manifolds satisfying the curvature condition \textbf{(C)}.

\section{Proof of the main result}

The energy functional associated to problem $(\mathscr{P}_{\lambda})$
is the functional $\mathcal{E}:H^1_{V}\to\mathbb{R}$ defined by
\[
\mathcal{E}(u)=\frac{1}{2}\|u\|_{V}^{2}-\lambda\int_{M}\alpha(x)F(u)\d,
\]
which is of class $C^1$ in $H^1_{V}$ with derivative, at any $u\in H^1_{V}$,
given by
\[
\mathcal{E}'(u)(v)=\int_{M}(\nabla_{g}u\nabla_{g}v+V(x)uv)\d-\lambda\int_{M}\alpha(x)f(u)v\d,\ \ \mbox{for all}\ v\in H^1_{V}.
\]
Weak solutions of problem $\Pl$ are precisely critical points of $\mathcal E$.

Because of the sign  of $f$, it is clear that critical points
of $\E$ are non negative functions. More properties of critical points
of $\E$ can be deduced by the following regularity theorem which
is crucial in the proof of the Theorem \ref{thm:main}. We adapt to our setting the classical Nash Moser iteration techniques.
\begin{thm}
\label{regularity}Let $N\geq3$ and $(M,g)$ be a complete, non-compact
$N-$dimensional Riemannian manifold satisfying the curvature condition $(\mathbf{C})$, and ${\displaystyle \inf_{x\in M}\mathrm{Vol}_{g}(B_{x}(1))>0}.$
Let also $\varphi:M\times\R_{+}\to\R$ be a continuous function with primitive
$\ds \Phi(x,t)=\int_{0}^{t}\varphi(x,\xi)d\xi$ such that, for some constants
$k>0$ and $q\in (2, 2^*)$ one has
\begin{align*}
 & |\varphi(x,\xi)|\leq k(\xi+\xi^{q-1}),\ \ \mbox{for all}\ \xi\geq0,\ \ \mbox{uniformly in }x\in M.
\end{align*}
Let $u\in H^1_{V}(M)$ be a non negative critical point of the functional $\mathcal{G}:H_{V}\to\R$
\[
\mathcal{G}(u)=\frac{1}{2}\|u\|_{V}^{2}-\int_{\R^{N}}\Phi(x,u)\d.
\] and $x_0\in M$.
Then,
\begin{itemize}
	\item[(i)] for every $\rho>0$, $u\in L^\infty(B_{x_0}(\rho))$;
	\item[(ii)] $u\in L^\infty(M)$ and $\ds\lim_{d_{g}(x_{0},x)\to\infty}u(x)=0$.
\end{itemize}

\end{thm}
\begin{proof}
Let $u$ be a critical point of $\mathcal{G}$. Then,
\begin{equation}
\int_{M}\left(\nabla_{g}u\nabla_{g}v+V(x)uv\right)\d=\int_{M}\varphi(x,u)v\d\qquad\mbox{for all}\ v\in H_{V}.\label{var}
\end{equation}
For each $L>0$, define
\[
u_{L}(x)=\left\{ \begin{array}{lll}
u(x) & \mbox{ if \ensuremath{u(x)\leq L}},\\
L & \mbox{ if \ensuremath{u(x)>L}}.
\end{array}\right.
\]
Let also $\tau\in C^{\infty}(M)$ with $0\leq\tau\leq1$.

For $\beta>1$, set $v_{L}=\tau^{2}uu_{L}^{2(\beta-1)}$ and $w_{L}=\tau uu_{L}^{\beta-1}$ which are in  $H^1_{V}(M)$. Thus, plugging $v_{L}$ into \eqref{var},
we get
\begin{equation}
\int_{M}\left(\nabla_{g}u\nabla_{g}v_{L}+V(x)uv_{L}\right)\d=\int_{M}\varphi(x,u)v_{L}\d.\label{one}
\end{equation}
 A direct calculation yields that
\[
\nabla_{g}v_{L}=2\tau\ u\ u_{L}^{2(\beta-1)}\ \nabla_{g}\tau+\tau^{2}\ u_{L}^{2(\beta-1)}\ \nabla_{g}u\ +2(\beta-1)\tau^{2}\ u\ u_{L}^{2\beta-3}\ \nabla_{g}u_{L},
\]
 and
\begin{align}
\begin{split}\int_{M}\nabla_{g}u\nabla_{g}v_{L}\d & =\int_{M}\left[2\tau\ u\ u_{L}^{2(\beta-1)}\ \nabla_{g}u\nabla_{g}\tau+\tau^{2}\ u_{L}^{2(\beta-1)}\ |\nabla_{g}u|^{2}\right]\d\label{zero}\\
 & +\int_{M}2(\beta-1)\tau^{2}\ u\ u_{L}^{2\beta-3}\ \nabla_{g}u\nabla_{g}u_{L}\d\\
 & \geq\int_{M}\left[2\tau\ u\ u_{L}^{2(\beta-1)}\nabla_{g}u\nabla_{g}\tau\ +\tau^{2}u_{L}^{2(\beta-1)}|\nabla_{g}u|^{2}\right]\d,
\end{split}
\end{align}
since
\[
2(\beta-1)\int_{M}\tau^{2}\ u\ u_{L}^{2\beta-3}\nabla_{g}u_{L}\nabla_{g}u\d=\int_{\{u\leq L\}}\tau^{2}\ u_{L}^{2(\beta-1)}|\nabla_{g}u|^{2}\d\geq0.
\]
Notice that
\begin{align*}
|\nabla_{g}w_{L}|^{2} & =u^{2}\ u_{L}^{2(\beta-1)}\ |\nabla_{g}\tau|^{2}+\tau^{2}\ u_{L}^{2(\beta-1)}\ |\nabla_{g}u|^{2}+(\beta-1)^{2}\ \tau^{2}u^{2}\ u_{L}^{2(\beta-2)}|\nabla_{g}u_{L}|^{2}\\
 & +2\tau\ u\ u_{L}^{2(\beta-1)}\ \nabla_{g}\tau\nabla_{g}u+2(\beta-1)\tau\ u^{2}\ u_{L}^{2\beta-3}\nabla_{g}\tau\nabla_{g}u_{L}+2(\beta-1)\ \tau^{2}\ u\ u_{L}^{2\beta-3}\nabla_{g}u\nabla_{g}u_{L}.
\end{align*}
Then, one can observe that
\[
\int_{M}\tau^{2}\ u^{2}\ u_{L}^{2(\beta-2)}|\nabla_{g}u_{L}|^{2}\d=\int_{\{u\leq L\}}\ \tau^{2}\ u_{L}^{2(\beta-1)}|\nabla_{g}u|^{2}\d\leq\int_{M}\ \tau^{2}\ u_{L}^{2(\beta-1)}|\nabla_{g}u|^{2}\d,
\]
 and
\[
\int_{M}\ \tau^{2}u\ u_{L}^{2\beta-3}\nabla_{g}u\nabla_{g}u_{L}\d=\int_{\{u\leq L\}}\ \tau^{2}\ u_{L}^{2(\beta-1)}|\nabla_{g}u|^{2}\d\leq\int_{M}\ \tau^{2}\ u_{L}^{2(\beta-1)}|\nabla_{g}u|^{2}\d,
\]
 and also that
\begin{align*}
2\int_{M}\ \tau\ u^{2}\ u_{L}^{2\beta-3}\ \nabla_{g}\tau\nabla_{g}u_{L}\d & \leq2\int_{M}\tau\ u^{2}\ u_{L}^{2\beta-3}\ |\nabla_{g}\tau|\cdot|\nabla_{g}u_{L}|\d\\
 & =2\int_{M}(\tau\ u\ u_{L}^{\beta-2}\ |\nabla_{g}u_{L}|)\cdot(u\ u_{L}^{\beta-1}\ |\nabla_{g}\tau|)\d\\
 & \leq\int_{M}\tau^2\ u^{2}\ u_{L}^{2(\beta-2)}\ |\nabla_{g}u_{L}|^{2}\d+\int_{M}u^{2}\ u_{L}^{2(\beta-1)}\ |\nabla_{g}\tau|^{2}\d\\
 & \leq\int_{M}\tau^{2}\ u_{L}^{2(\beta-1)}|\nabla_{g}u|^{2}\d+\int_{M}u^{2}\ u_{L}^{2(\beta-1)}\ |\nabla_{g}\tau|^{2}\d.
\end{align*}
Therefore
\begin{eqnarray}
\begin{split}\int_{M}|\nabla_{g}w_{L}|^{2}\d & \leq\int_{M}\ u^{2}\ u_{L}^{2(\beta-1)}|\nabla_{g}\tau|^{2}\d+\beta^{2}\int_{M}\tau^{2}\ u_{L}^{2(\beta-1)}|\nabla_{g}u|^{2}\d+\\
 & +2\int_{M}\tau\ u\ u_{L}^{2(\beta-1)}\ \nabla_{g}\tau\nabla_{g}u\d+2(\beta-1)\int_{M}\tau\ u^{2}\ u_{L}^{2\beta-3}\ \nabla_{g}\tau\nabla u_{L}\d\\
 & \leq\beta\int_{M}u^{2}\ u_{L}^{2(\beta-1)}|\nabla_{g}\tau|^{2}\d+(\beta^{2}+\beta-1)\int_{M}\tau^{2}\ u_{L}^{2(\beta-1)}|\nabla_{g}u|^{2}\d+\label{three}\\
 & +2\int_{M}\tau\ u\ u_{L}^{2(\beta-1)}\ \nabla_{g}\tau\nabla_{g}u\d.
\end{split}
\end{eqnarray}
In the sequel we will need the constant $\gamma=\frac{2 \cdot 2^\star}{2^\star-q+2}$. It is clear that  $2<\gamma<2^\star$. \\
\emph{Proof of $i)$. }Putting together (\ref{zero}), \eqref{three}, with (\ref{one}), recalling that $\beta>1$, and bearing in mind the growth of the function $\varphi$,
we obtain that
\begin{eqnarray*}
\|w_{L}\|_{V}^{2} & = & \int_{M}\left(|\nabla_{g}w_{L}|^{2}+V(x)w_{L}^{2}\right)\d\\
 & \leq & \beta\int_{M}\ u^{2}\ u_{L}^{2(\beta-1)}|\nabla_{g}\tau|^{2}\d+2\beta^{2}\int_{M}\left(\nabla_{g}u\nabla_{g}v_{L}+V(x)\tau^{2}u^{2}\ u_{L}^{2(\beta-1)}\right)\d\\
 & = & \beta\int_{M}\ u^{2}\ u_{L}^{2(\beta-1)}|\nabla_{g}\tau|^{2}\d+2\beta^{2}\int_{M}\left(\nabla_{g}u\nabla_{g}v_{L}+V(x)uv_{L}\right)\d\\
 & = & \beta\int_{M}\ u^{2}\ u_{L}^{2(\beta-1)}|\nabla_{g}\tau|^{2}\d+2\beta^{2}\int_{M}\varphi(x,u)v_{L}\d\\
 & \leq & \beta\int_{M}\ u^{2}\ u_{L}^{2(\beta-1)}|\nabla_{g}\tau|^{2}\d+2\beta^{2}k\int_{M}\left(\tau^{2}u^{2}\ u_{L}^{2(\beta-1)}+\tau^{2}u^{q}u_{L}^{2(\beta-1)}\right)\d\\
 & = & \beta\underbrace{\int_{M}\ u^{2}\ u_{L}^{2(\beta-1)}|\nabla_{g}\tau|^{2}\d}_{I_{1}}+2\beta^{2}k\underbrace{\int_{M}w_{L}^{2}\d}_{I_{2}}+2\beta^{2}k\underbrace{\int_{M}u^{q-2}w_{L}^{2}\d}_{I_{3}}.
\end{eqnarray*}
Let $R,r>0$. In the proof of case $i)$, $\tau$ verifies the further
following properties: $|\nabla\tau|\leq\frac{2}{r}$ and
\[
\tau(x)=\left\{ \begin{array}{lll}
1 & \mbox{ if \ensuremath{d_g(x_0,x)\leq R}},\\
0 & \mbox{ if \ensuremath{d_g(x_0,x)>R+r}}.
\end{array}\right.
\]
Then, applying  Hölder inequality yields that
\begin{eqnarray*}
I_{1} & \leq & \frac{4}{r^{2}}\int_{R\leq d_{g}(x_{0},x)\leq R+r}u^{2}\ u_{L}^{2(\beta-1)}\d\\
 & \leq & \frac{4}{r^{2}}\left(\mathrm{Vol}_{g}\left(A[R,R+r]\right)\right)^{1-\frac{2}{\gamma}}\left(\int_{A[R,R+r]}u^{\gamma}\ u_{L}^{\gamma(\beta-1)}\d\right)^{\frac{2}{\gamma}},
\end{eqnarray*}
where $A[R,R+r]=\left\{ x\in M:\ R\leq d_{g}(x_{0},x)\leq R+r\right\} $. Then, from Theorem \ref{A}, we have that
\begin{align*}
I_{1} & \leq4\omega_{N}^{1-\frac{2}{\gamma}}e^{(N-1)b_{0}\left(1-\frac{2}{\gamma}\right)}\frac{(R+r)^{N(1-\frac{2}{\gamma})}}{r^{2}}\left(\int_{d_{g}(x_{0},x)\leq R+r}u^{\gamma}\ u_{L}^{\gamma(\beta-1)}\d\right)^{\frac{2}{\gamma}}.
\end{align*}
In a similar way, we obtain that
\begin{eqnarray*}
I_{2} & \leq & \int_{d_{g}(x_{0},x)\leq R+r}u^{2}\ u_{L}^{2(\beta-1)}\d\\
 & \leq & \omega_{N}^{1-\frac{2}{\gamma}}e^{(N-1)b_{0}\left(1-\frac{2}{\gamma}\right)}\ (R+r)^{N(1-\frac{2}{\gamma})}\left(\int_{d_{g}(x,x_{0})\leq R+r}u^{\gamma}\ u_{L}^{\gamma(\beta-1)}\d\right)^{\frac{2}{\gamma}},
\end{eqnarray*}
and also that
\begin{eqnarray*}
I_{3} & = & \int_{M}u^{q-2}w_{L}^{2}\d \leq  \left(\int_{M}u^{2^{\star}}\d\right)^{\frac{q-2}{2^{\star}}}\left(\int_{M}w_{L}^{\gamma}\d\right)^{\frac{2}{\gamma}}\\
 & = & \|u\|_{L^{2^{\star}}(M)}^{q-2}\left(\int_{d_{g}(x_{0},x)\leq R+r}u^{\gamma}\ u_{L}^{\gamma(\beta-1)}\d\right)^{\frac{2}{\gamma}}.
\end{eqnarray*}
In the sequel we will use the  notation $\displaystyle \mathscr{J}=\left(\int_{d_{g}(x_{0},x)\leq R+r}u^{\gamma}\ u_{L}^{\gamma(\beta-1)}\d\right)^{\frac{2}{\gamma}}$. Therefore, summing up the above computations, we obtain that
\begin{eqnarray}
\|w_{L}\|_{V}^{2} &\leq &  4\beta\omega_{N}^{1-\frac{2}{\gamma}}e^{(N-1)b_{0}\left(1-\frac{2}{\gamma}\right)} \frac{(R+r)^{N(1-\frac{2}{\gamma})}}{r^{2}}\mathscr{J}+2\beta^{2}k\|u\|_{L^{2^{\star}}(M)}^{q-2}\mathscr{J}\nonumber \\&+&2\beta^{2}k\omega_{N}^{1-\frac{2}{\gamma}}e^{(N-1)b_{0}\left(1-\frac{2}{\gamma}\right)}\ (R+r)^{N(1-\frac{2}{\gamma})}\mathscr{J}.\label{ineq1}
\end{eqnarray}
Moreover, if $C_{\star}$ denotes the embedding constant of $H^1_{V}(M)$, one has
into $L^{2^{\star}}(M)$,
\[
\|w_{L}\|_{V}^{2}\geq C_{\star}\|w_{L}\|_{L^{2^{\star}}(M)}^{2}=C_{\star}\left(\int_{M}(\tau\ u\ u_{L}^{\beta-1})^{2^{\star}}\d\right)^{\frac{2}{2^{\star}}}\geq C_{\star}\left(\int_{d_{g}(x_{0},x)\leq R}(u\ u_{L}^{\beta-1})^{2^{\star}}\d\right)^{\frac{2}{2^{\star}}}.
\]
Combining the above computations with \eqref{ineq1}, and bearing
in mind that $\beta>1$, we get
\begin{eqnarray}\label{ineq2}
\left(\int_{d_{g}(x_{0},x)\leq R}(u\ u_{L}^{\beta-1})^{2^{\star}}\d\right)^{\frac{2}{2^{\star}}} &\leq & 4C_{\star}^{-1}\beta^{2}\omega_{N}^{1-\frac{2}{\gamma}}e^{(N-1)b_{0}\left(1-\frac{2}{\gamma}\right)}\ \frac{(R+r)^{N(1-\frac{2}{\gamma})}}{r^{2}}\mathscr{J}+2kC_{\star}^{-1}\beta^{2}\|u\|_{2^{\star}}^{q-2}\mathscr{J}\nonumber\\ &+&2kC_{\star}^{-1}\beta^{2}\omega_{N}^{1-\frac{2}{\gamma}}e^{(N-1)b_{0}\left(1-\frac{2}{\gamma}\right)}\ (R+r)^{N(1-\frac{2}{\gamma})}\mathscr{J}.
\end{eqnarray}
Taking the limit as $L\to+\infty$ in (\ref{ineq2}), we obtain

\begin{align*}
\left(\int_{d_{g}(x_{0},x)\leq R}u^{2^{\star}\beta}\d\right)^{\frac{2}{2^{\star}}} \leq& 4C_{\star}^{-1}\beta^{2}\omega_{N}^{1-\frac{2}{\gamma}}\ \frac{(R+r)^{N(1-\frac{2}{\gamma})}}{r^{2}}\left(\int_{d_g(x_0,x)\leq R+r}u^{\gamma\beta}\right)^{\frac{2}{\gamma}}+\\ &+2kC_{\star}^{-1}\beta^{2}\omega_{N}^{1-\frac{2}{\gamma}} (R+r)^{N(1-\frac{2}{\gamma})}\left(\int_{d_g(x_0,x)\leq R+r}u^{\gamma\beta}\right)^{\frac{2}{\gamma}}+\\&+2C_{\star}^{-1}\beta^{2}k\|u\|_{2^{\star}}^{q-2}\left(\int_{d_g(x_0,x)\leq R+r}u^{\gamma\beta}\right)^{\frac{2}{\gamma}}.
\end{align*}
Thus, for every $R>0$, $r>0$, $\beta>1$ one has
\begin{equation}\small
\|u\|_{L^{2^{\star}\beta}(d_{g}(x_{0},x)\leq R)}\leq(C_{\star}^{-1})^{\frac{1}{2\beta}}\beta^{\frac{1}{\beta}}\left(C_{1}\ \frac{(R+r)^{N(1-\frac{2}{\gamma})}}{r^{2}}+C_{2}\ (R+r)^{N(1-\frac{2}{\gamma})}+C_{3}\right)^{\frac{1}{2\beta}}\|u\|_{L^{\gamma\beta}(d_{g}(x_{0},x)\leq R+r)},\label{ineq}
\end{equation}
where $C_{1}=4\omega_{N}^{1-\frac{2}{\gamma}}e^{(N-1)b_{0}\left(1-\frac{2}{\gamma}\right)},\ C_{2}=2k\omega_{N}^{1-\frac{2}{\gamma}}e^{(N-1)b_{0}\left(1-\frac{2}{\gamma}\right)},\ C_{3}=2k\|u\|_{L^{2^{*}}(M)}^{q-2}$.

Fix $\rho>0$. We are going to apply (\ref{ineq}) choosing first
${\displaystyle \beta=\frac{2^{\star}}{\gamma},\ R=\rho+\frac{\rho}{2},\ r=\frac{\rho}{2},}$
to get
$${\scriptsize
\|u\|_{L^{2^{\star}\beta}(d_{g}(x_{0},x)\leq\rho+\frac{\rho}{2})}\leq(C_{\star}^{-1})^{\frac{1}{2\beta}}\beta^{\frac{1}{\beta}}\left(C_{1}\ 2^{N(1-\frac{2}{\gamma})}\rho^{N(1-\frac{2}{\gamma})-2}2^{2}+C_{2}\ (2\rho)^{N(1-\frac{2}{\gamma})}+C_{3}\right)^{\frac{1}{2\beta}}\|u\|_{L^{2^{\star}}(d_{g}(x_{0},x)\leq2\rho)}}
$$
Noticing that $\gamma\beta^{2}=2^{\star}\beta$, we can  apply (\ref{ineq})
with $\beta^{2}$ in place of $\beta$ and $R=\rho+\frac{\rho}{2^{2}},\ r=\frac{\rho}{2^{2}}.$
We obtain

\begin{eqnarray*}
\|u\|_{L^{2^{\star}\beta^{2}}(d_{g}(x_{0},x)\leq\rho+\frac{\rho}{2^{2}})}  & \leq & (C_{\star}^{-1})^{\frac{1}{2\beta}+\frac{1}{2\beta^{2}}}\beta^{\frac{1}{\beta}+\frac{2}{\beta^{2}}}e^{\frac{1}{2\beta}\log\left(C_{1}\ 2^{N(1-\frac{2}{\gamma})}\rho^{N(1-\frac{2}{\gamma})-2}2^{2}+C_{2}\ (2\rho)^{N(1-\frac{2}{\gamma})}+C_{3}\right)}\cdot\\
 &  & \cdot e^{\frac{1}{2\beta^{2}}\log\left(C_{1}\ 2^{N(1-\frac{2}{\gamma})}\rho^{N(1-\frac{2}{\gamma})-2}(2^{2})^{2}+C_{2}\ (2\rho)^{N(1-\frac{2}{\gamma})}+C_{3}\right)}\|u\|_{L^{2^{\star}}(d_{g}(x_{0},x)\leq2\rho)}
\end{eqnarray*}
Iterating this procedure, for every integer $n$ we obtain

\begin{eqnarray*}
 &&\|u\|_{L^{2^{\star}\beta^{n}}(d_{g}(x_{0},x)\leq\rho)}\leq\|u\|_{L^{2^{\star}\beta^{n}}(d_{g}(x_{0},x)\leq\rho+\frac{\rho}{2^{n}})}
 \\&& \leq  (C_{\star}^{-1})^{\sum_{i=1}^{n}\frac{1}{2\beta^{i}}}\beta^{\sum_{i=1}^{n}\frac{i}{\beta^{i}}}
 e^{\sum_{i=1}^{n}\frac{\log\left(C_{1}\ 2^{N(1-\frac{2}{\gamma})}\rho^{N(1-\frac{2}{\gamma})-2}2^{2i}
 		+C_{2}(2\rho)^{N(1-\frac{2}{\gamma})}+C_{3}\right)}{2\beta^{i}}}\|u\|_{L^{2^{\star}}(d_{g}(x_{0},x)\leq2\rho)}.
\end{eqnarray*}
If
\[\small
{\displaystyle \sigma=\frac{1}{2}\sum_{n=1}^{\infty}\frac{1}{\beta^{n}}=\frac{1}{2(\beta-1)},\;\vartheta=\sum_{n=1}^{\infty}\frac{n}{\beta^{n}}},\ \eta=\sum_{n=1}^{\infty}\frac{\log\left(C_{1}\ 2^{N(1-\frac{2}{\gamma})}\rho^{N(1-\frac{2}{\gamma})-2}2^{2n}+C_{2}\ (2\rho)^{N(1-\frac{2}{\gamma})}+C_{3}\right)}{2\beta^{n}}.
\]
Passing to the limit as $n\to\infty$, we obtain
\[
\|u\|_{L^{\infty}(d_{g}(x_{0},x)\leq\rho)}\leq(C_{\star}^{-1})^{\sigma}\beta^{\vartheta}e^{\eta}\|u\|_{L^{2^{\star}}(d_{g}(x_{0},x)\leq2\rho)}.
\]
Since $u\in L^{2^{\star}}(M)$, claim $(i)$ follows at once. Notice
that $\eta$ depends on $\rho$.

\noindent \textit{Proof of $(ii)$}. Since $V$ is coercive, we can find $\bar{R}>0$
such that
\[
V(x)\geq2k\qquad\mbox{for}\qquad d_g(x_0,x)\geq\bar{R}
\]
(where $k$ is from the growth of $\varphi$. Without loss of generality
we can assume that $k\geq1$.)

Let $R>\max\{\bar{R},1\}$, $0<r\leq\frac{R}{2}$. In the proof of
case $ii)$, $\tau$ verifies the further following properties: $|\nabla\tau|\leq\frac{2}{r}$
and $\tau$ is such that
\[
\tau(x)=\left\{ \begin{array}{lll}
0 & \mbox{ if \ensuremath{d_g(x_0,x)\leq R}},\\
1 & \mbox{ if \ensuremath{d_g(x_0,x)>R+r}}.
\end{array}\right.
\]
From (\ref{one}), we get
\begin{eqnarray*}
\int_{M}(\nabla_{g}u\nabla_{g}v_{L}+2kuv_{L})\d & = & \int_{d_{g}(x_{0},x)\geq R}(\nabla_{g}u\nabla_{g}v_{L}+2kuv_{L})\d\\&\leq&\int_{d_{g}(x_{0},x)\geq R}(\nabla_{g}u\nabla_{g}v_{L}+V(x)uv_{L})\d\\
 & = & \int_{M}(\nabla_{g}u\nabla_{g}v_{L}+V(x)uv_{L})\d=\int_{M}\varphi(x,u)v_{L}\d\\
 & \leq & k\int_{M}(uv_{L}+u^{q-1}v_{L})\d,
\end{eqnarray*}
thus,
\[
\int_{M}(\nabla_{g}u\nabla_{g}v_{L}+uv_{L})\d\leq\int_{M}(\nabla_{g}u\nabla_{g}v_{L}+kuv_{L})\d\leq k\int_{M}u^{q-1}v_{L}\d.
\]
From (\ref{zero}) and (\ref{three}), and since $w_{L}^{2}=u\cdot v_{L}$,
\begin{eqnarray*}
\int_{M}(|\nabla_{g}w_{L}|^{2}+w_{L}^{2})\d & \leq & \beta\int_{M}\ u^{2}\ u_{L}^{2(\beta-1)}|\nabla_{g}\tau|^{2}\d+2\beta^{2}\int_{M}\nabla_{g}u\nabla_{g}v_{L}\d+\int_{M}uv_{L}\d\\
 & \leq & \beta\int_{M}\ u^{2}\ u_{L}^{2(\beta-1)}|\nabla_{g}\tau|^{2}\d+2\beta^{2}\int_{M}(\nabla_{g}u\nabla_{g}v_{L}+uv_{L})\d\\
 & \leq & \beta\int_{M}\ u^{2}\ u_{L}^{2(\beta-1)}|\nabla_{g}\tau|^{2}\d+2\beta^{2}k\int_{M}u^{q-1}v_{L}\d.
\end{eqnarray*}
Thus,
\[
\|w_{L}\|_{H_g^{1}(M)}^{2}\leq\beta\underbrace{\int_{M}\ u^{2}\ u_{L}^{2(\beta-1)}|\nabla_{g}\tau|^{2}\d}_{I_{1}}+2\beta^{2}k\underbrace{\int_{M}u^{q-2}w_{L}^{2}\d}_{I_{2}}.
\]
As in the proof of $i)$ one has
\begin{eqnarray*}
I_{1} & \leq & 4\omega_{N}^{1-\frac{2}{\gamma}}e^{(N-1)b_{0}\left(1-\frac{2}{\gamma}\right)}\ \frac{(R+r)^{N(1-\frac{2}{\gamma})}}{r^{2}}\left(\int_{d_{g}(x_{0},x)\geq R}u^{\gamma}\ u_{L}^{\gamma(\beta-1)}\d\right)^{\frac{2}{\gamma}},
\end{eqnarray*}
and
\begin{eqnarray*}
I_{2} & \leq & \|u\|_{L^{2^{\star}}(M)}^{q-2}\left(\int_{d_{g}(x_{0},x)\geq R}u^{\gamma}\ u_{L}^{\gamma(\beta-1)}\d\right)^{\frac{2}{\gamma}}.
\end{eqnarray*}
Since,
\[
\|w_{L}\|_{H_g^{1}(M)}^{2}\geq C^{\star}\|w_{L}\|_{L^{2^{\star}}(M)}^{2}=C^{\star}\left(\int_{M}(\tau\ u\ u_{L}^{\beta-1})^{2^{\star}}\d\right)^{\frac{2}{2^{\star}}}\geq C^{\star}\left(\int_{d_{g}(x_{0},x)\geq R+r}(u\ u_{L}^{\beta-1})^{2^{\star}}\d\right)^{\frac{2}{2^{\star}}},
\]
where $C^{\star}$ denotes the embedding constant of $H_g^{1}(M)$ into
$L^{2^{\star}}(M)$, we obtain
\begin{align*}
\left(\int_{d_{g}(x_{0},x)\geq R+r}(u\ u_{L}^{\beta-1})^{2^{\star}}\d\right)^{\frac{2}{2^{\star}}}\leq&4(C^{\star})^{-1}\beta^{2}\omega_{N}^{1-\frac{2}{\gamma}}e^{(N-1)b_{0}\left(1-\frac{2}{\gamma}\right)}\ \frac{(R+r)^{N(1-\frac{2}{\gamma})}}{r^{2}}\cdot \\ &\cdot\left(\int_{d_{g}(x_{0},x)\geq R}u^{\gamma}\ u_{L}^{\gamma(\beta-1)}\d\right)^{\frac{2}{\gamma}}+\\&+2(C^{\star})^{-1}\beta^{2}k\|u\|_{L^{2^{\star}}(M)}^{q-2}\cdot\left(\int_{d_{g}(x_{0},x)\geq R}u^{\gamma}\ u_{L}^{\gamma(\beta-1)}\d\right)^{\frac{2}{\gamma}}.
\end{align*}
Taking the limit as $L\to+\infty$ in the above inequality, we obtain
\begin{align*}
\left(\int_{d_{g}(x_{0},x)\geq R+r}u^{2^{\star}\beta}\d\right)^{\frac{2}{2^{\star}}}\leq& 4(C^{\star})^{-1}\beta^{2}\omega_{N}^{1-\frac{2}{\gamma}}e^{(N-1)b_{0}\left(1-\frac{2}{\gamma}\right)} \frac{(R+r)^{N(1-\frac{2}{\gamma})}}{r^{2}}\left(\int_{d_{g}(x_{0},x)\geq R}u^{\gamma}\ u_{L}^{\gamma(\beta-1)}\d\right)^{\frac{2}{\gamma}}+\\&+2(C^{\star})^{-1}\beta^{2}k\|u\|_{2^{\star}}^{q-2}\left(\int_{d_{g}(x_{0},x)\geq R}u^{\gamma}\ u_{L}^{\gamma(\beta-1)}\d\right)^{\frac{2}{\gamma}}.
\end{align*}
Thus, for every $R>\max\{\bar{R},1\}$, $0<r\leq\frac{R}{2}$, $\beta>1$
one has
\begin{equation}
\|u\|_{L^{2^{\star}\beta}(d_{g}(x_{0},x)\geq R+r)}\leq((C^{\star})^{-1})^{\frac{1}{2\beta}}\beta^{\frac{1}{\beta}}\left(C_{1}\ \frac{(R+r)^{N(1-\frac{2}{\gamma})}}{r^{2}}+C_{2}\right)^{\frac{1}{2\beta}}\|u\|_{L^{\gamma\beta}(d_{g}(x_{0},x)\geq R)}\label{ineq 5},
\end{equation}
where $C_{1}=4\omega_{N}^{1-\frac{2}{\gamma}}e^{(N-1)b_{0}\left(1-\frac{2}{\gamma}\right)},\ C_{2}=2k\|u\|_{L^{2^{\star}}(M)}^{q-2}$.
Fix $\rho>\max\{\bar{R},1\}$. We are going to apply (\ref{ineq 5})
choosing first ${\displaystyle \beta=\frac{2^{\star}}{\gamma},\ R=\rho+\frac{\rho}{2},\ r=\frac{\rho}{2},}$
to get
\[
\|u\|_{L^{2^{\star}\beta}(d_{g}(x_{0},x)\geq2\rho)}\leq((C^{\star})^{-1})^{\frac{1}{2\beta}}\beta^{\frac{1}{\beta}}\left(C_{1}\ 2^{N(1-\frac{2}{\gamma})}\rho^{N(1-\frac{2}{\gamma})-2}2^{2}+C_{2}\right)^{\frac{1}{2\beta}}\|u\|_{L^{2^{\star}}(d_{g}(x_{0},x)\geq\rho+\frac{\rho}{2})}.
\]
Noticing that $\gamma\beta^{2}=2^{\star}\beta$, let us apply (\ref{ineq 5})
with $\beta^{2}$ in place of $\beta$ and $R=\rho+\frac{\rho}{2^{2}},\ r=\frac{\rho}{2^{2}},$
to obtain

\begin{eqnarray*}
\|u\|_{L^{2^{\star}\beta^{2}}(d_{g}(x_{0},x)\geq\rho+\frac{\rho}{2})} & \leq & ((C^{\star})^{-1})^{\frac{1}{2\beta^{2}}}\beta^{\frac{2}{\beta^{2}}}\left(C_{1}\ \frac{(\rho+\frac{\rho}{2})^{N(1-\frac{2}{\gamma})}}{\rho^{2}}(2^{2})^{2}+C_{2}\right)^{\frac{1}{2\beta^{2}}}\|u\|_{L^{2^{\star}\beta}(d_{g}(x_{0},x)\geq\rho+\frac{\rho}{2^{2}})}\\
 & \leq & ((C^{\star})^{-1})^{\frac{1}{2\beta^{2}}}\beta^{\frac{2}{\beta^{2}}}\left(C_{1}\ 2^{N(1-\frac{2}{\gamma})}\rho^{N(1-\frac{2}{\gamma})-2}(2^{2})^{2}+C_{2}\right)^{\frac{1}{2\beta^{2}}}\|u\|_{L^{2^{\star}\beta}(d_{g}(x_{0},x)\geq\rho+\frac{\rho}{2^{2}})}
\end{eqnarray*}
Thus, combining the previous two inequalities we get
\begin{eqnarray*}
\|u\|_{L^{2^{\star}\beta}(d_{g}(x_{0},x)\geq2\rho)} & \leq & ((C^{\star})^{-1})^{\frac{1}{2\beta}+\frac{1}{2\beta^{2}}}\beta^{\frac{1}{\beta}+\frac{2}{\beta^{2}}}e^{\frac{1}{2\beta}\log\left(C_{1}\ 2^{N(1-\frac{2}{\gamma})}\rho^{N(1-\frac{2}{\gamma})-2}2^{2}+C_{2}\right)}\cdot\\
 &  & \cdot e^{\frac{1}{2\beta^{2}}\log\left(C_{1}\ 2^{N(1-\frac{2}{\gamma})}\rho^{N(1-\frac{2}{\gamma})-2}(2^{2})^{2}+C_{2}\right)}\|u\|_{L^{2^{\star}}(d_{g}(x_{0},x)\geq\rho+\frac{\rho}{2^{2}})}
\end{eqnarray*}
Iterating this procedure, for every integer $n$ we obtain

\begin{eqnarray*}
 &  & \|u\|_{L^{2^{\star}\beta^{n}}(d_{g}(x_{0},x)\geq2\rho)}\\
 & \leq & {\displaystyle ((C^{\star})^{-1})^{{\displaystyle \sum_{i=1}^{n}\frac{1}{2\beta^{i}}}}\cdot\beta^{{\displaystyle \sum_{i=1}^{n}\frac{i}{\beta^{i}}}}\cdot e^{{\displaystyle \sum_{i=1}^{n}\frac{\log\left(C_{1}\ 2^{N(1-\frac{2}{\gamma})}\rho^{N(1-\frac{2}{\gamma})-2}2^{2i}+C_{2}\right)}{2\beta^{i}}}}\|u\|_{L^{2^{\star}}(d_{g}(x_{0},x)\geq\rho+\frac{\rho}{2^{n}})}}\\
 & \leq & {\displaystyle ((C^{\star})^{-1})^{{\displaystyle \sum_{i=1}^{n}\frac{1}{2\beta^{i}}}}\cdot\beta^{{\displaystyle \sum_{i=1}^{n}\frac{i}{\beta^{i}}}}\cdot e^{{\displaystyle \sum_{i=1}^{n}\frac{\log\left(C_{1}\ 2^{N(1-\frac{2}{\gamma})}\rho^{N(1-\frac{2}{\gamma})-2}2^{2i}+C_{2}\right)}{2\beta^{i}}}}\|u\|_{L^{2^{\star}}(d_{g}(x_{0},x)\geq\rho)}.}
\end{eqnarray*}
Since $N(1-\frac{2}{\gamma})<2$ and $\rho>1$, one has $\rho^{N(1-\frac{2}{\gamma})-2}<1$,
and the previous estimate implies

\[
\|u\|_{L^{2^{\star}\beta^{n}}(d_{g}(x_{0},x)\geq2\rho)}\leq{\displaystyle ((C^{\star})^{-1})^{{\displaystyle \sum_{i=1}^{n}\frac{1}{2\beta^{i}}}}\cdot\beta^{{\displaystyle \sum_{i=1}^{n}\frac{i}{\beta^{i}}}}\cdot e^{{\displaystyle \sum_{i=1}^{n}\frac{\log\left(C_{1}\ 2^{N(1-\frac{2}{\gamma})}2^{2i}+C_{2}\right)}{2\beta^{i}}}}\|u\|_{L^{2^{\star}}(d_{g}(x_{0},x)\geq\rho)}}.
\]
If
\[
{\displaystyle \sigma=\frac{1}{2}\sum_{n=1}^{\infty}\frac{1}{\beta^{n}}=\frac{1}{2(\beta-1)},\;\vartheta=\sum_{n=1}^{\infty}\frac{n}{\beta^{n}}},\ \zeta=\sum_{n=1}^{\infty}\frac{\log\left(C_{1}\ 2^{N(1-\frac{2}{\gamma})}2^{2n}+C_{2}\right)}{2\beta^{n}},
\] passing to the limit as $n\to\infty$,
we obtain

\[
\|u\|_{L^{\infty}(d_{g}(x_{0},x)\geq2\rho)}\leq C_{0}\|u\|_{L^{2^{\star}}(d_{g}(x_{0},x)\geq\rho)}
\] where $C_{0}=(C^{\star})^{-\sigma}\beta^{\vartheta}e^{\zeta}$
does not depend on $\rho$.
\bigskip{}
Taking into account that $u\in L^{2^{\star}}(M)$, and combining the
above inequality with claim $i)$, we obtain that $u\in L^{\infty}(M)$.
Moreover, as $\displaystyle \lim_{\rho\to\infty}\|u\|_{L^{2^{\star}}(d_g(x_0,x)\geq\rho)}=0,$
we deduce also that $\ds \lim_{d_g(x_0,x)\to\infty}u(x)=0$.
\end{proof}

\bigskip

Now, we consider the following minimization problem:\vspace{0.5cm}\\
\noindent {\bf (M)} $\min\left\{\|u\|_{V}^{2}:\ u\in H^1_{V}(M),\ \|\alpha^{\frac{1}{2}}u\|_{L^{2}(M)}=1\right\}.$ \\

\begin{lem}
Problem \textbf{(M)} has a non negative solution $\varphi_{\alpha}\in L^{\infty}(M)$
such that for every $x_0\in M$, ${\displaystyle \lim_{d_g(x_0,x)\to\infty}\varphi_{\alpha}(x)=0}$.
Moreover, $\varphi_{\alpha}$ is an eigenfunction of the equation
\[
-\Delta_{g}u+V(x)u=\lambda\alpha(x)u,\qquad u\in H^1_{V}(M)
\]
corresponding to the eigenvalue $\|\varphi_{\alpha}\|_{V}^{2}$.
\end{lem}
\begin{proof}
Notice first that $\alpha^{\frac{1}{2}}u\in L^{2}(M)$ for any $u\in H^1_{V}(M)$.
Fix a minimizing sequence $\{u_{n}\}$ for problem \textbf{(M)}, that
is $\|u_{n}\|_{V}^{2}\to\lambda_{\alpha}$, being
\[
\lambda_{\alpha}=\inf\left\{ \|u\|_{V}^{2}:\ u\in H^1_{V}(M),\|\alpha^{\frac{1}{2}}u\|_{L^{2}(M)}=1\right\} .
\]
Then, there exists a subsequence (still denoted by $\{u_n\}$) weakly converging in $H^1_{V}(M)$ to
some $\varphi_{\alpha}\in H^1_{V}(M)$. By the weak lower semicontinuity
of the norm, we obtain that
\[
\|\varphi_{\alpha}\|_{V}^{2}\leq\liminf_{n}\|u_{n}\|_{V}^{2}=\lambda_{\alpha}.
\]
In order to conclude, it is enough to prove that $\|\alpha^{\frac{1}{2}}\varphi_{\alpha}\|_{L^{2}(M)}=1.$
Since $\{u_{n}\}$ converges strongly to $\varphi_{\alpha}$ in $L^{2}(M)$
and $\alpha\in L^{\infty}(M)$,
\[
\alpha^{\frac{1}{2}}u_{n}\to\alpha^{\frac{1}{2}}\varphi_{\alpha}\qquad\mbox{in}\ L^{2}(M),
\]
thus, by the continuity of the norm, $\|\alpha^{\frac{1}{2}}\varphi_{\alpha}\|_{L^{2}(M)}=1$
and the claim is proved. Clearly, $\varphi_\alpha\neq 0$. Replacing eventually $\varphi_{\alpha}$
with $|\varphi_{\alpha}|$ we can assume that $\varphi_{\alpha}$
is non negative. Equivalently, we can write
\[
\lambda_{\alpha}=\inf_{u\in H^1_{V}(M)\setminus\{0\}}\dfrac{\|u\|_{V}^{2}}{\|\alpha^{\frac{1}{2}}u\|_{L^{2}(M)}^{2}}.
\]
This means that $\varphi_{\alpha}$ is a global minimum of the function
$u\to\dfrac{\|u\|_{V}^{2}}{\|\alpha^{\frac{1}{2}}u\|_{L^{2}(M)}^{2}}$,
hence its derivative at $\varphi_{\alpha}$ is zero, i.e.
\[
\int_{M}(\nabla_{g}\varphi_{\alpha}\nabla_{g}v+V(x)\varphi_{\alpha}v)\d-\|\varphi_{\alpha}\|_{V}^{2}\int_{M}\alpha(x)\varphi_{\alpha}v\d=0\ \mbox{for any}\ v\in H_{V}
\]
(recall that $\|\alpha^{\frac{1}{2}}\varphi_{\alpha}\|_{L^{2}(M)}=1$).
The above equality implies that $\varphi_{\alpha}$ is an eigenfunction
of the problem
\[
-\Delta_{g}u+V(x)u=\lambda\alpha(x)u,\qquad u\in H^1_{V}(M)
\]
corresponding to the eigenvalue $\|\varphi_{\alpha}\|_{V}^{2}$. From
Theorem \ref{regularity}  we also have that $\varphi_{\alpha}$
is a bounded function and ${\displaystyle \lim_{d_{g}(x,x_{0})\to\infty}\varphi_{\alpha}(x)=0}$.
\end{proof}

\bigskip

Now we are in the position to prove our main theorem.

\subsection{Proof of Theorem \ref{thm:main}}
$(i)\Rightarrow(ii)$. 

From the assumption,
we deduce the existence of $\sigma_{1}\in(0,+\infty]$
defined as
\[
\sigma_{1}\equiv\lim_{\xi\to0}\frac{{F}(\xi)}{\xi^{2}}.
\]

Assume first that $\sigma_1<\infty$. 

Define the following continuous truncation
of $f$,

\[
\tilde{f}(\xi)=\left\{ \begin{array}{ll}
0, & \mbox{\mbox{{if\;}}}\xi\in(-\infty,0]\\
\\
f(\xi), & \mbox{}\mbox{{if\;}}\xi\in(0,a]\\
\\
f(a), & \mbox{}\mbox{{if\;}}\xi\in(a,+\infty)
\end{array}\right.
\]
and let $\tilde{F}$ its primitive, that is $\tilde{F}(\xi)=\displaystyle\int_{0}^{\xi}\tilde{f}(t)dt$,
i.e.
\[
\tilde{F}(\xi)=\left\{ \begin{array}{ll}
F(\xi), & \mbox{{if\;}}\xi\in(-\infty,a]\\
\\
F(a)+f(a)(\xi-a), & \mbox{{if\;}}\xi\in(a,+\infty).
\end{array}\right.
\]
Observe that, from the monotonicity assumption on the function $\xi\to\frac{F(\xi)}{\xi^{2}}$,
the derivative of the latter is non-positive, that is
\[
f(\xi)\xi\leq2F(\xi)\qquad\mbox{for all}\ \xi\in[0,a].
\]
This implies
\begin{equation}
\tilde{f}(\xi)\xi\leq2\tilde{F}(\xi)\qquad\mbox{for all}\ \xi\in\R,\label{monotonicity2}
\end{equation}
or that the function $\xi\to\frac{\tilde{F}(\xi)}{\xi^{2}}$ is not
increasing in $(0,+\infty)$. Then, 
\begin{equation}
\sigma_{1}\equiv\lim_{\xi\to0}\frac{{F}(\xi)}{\xi^{2}}=\lim_{\xi\to0}\frac{\tilde{F}(\xi)}{\xi^{2}}=\sup_{\xi>0}\frac{\tilde{F}(\xi)}{\xi^{2}}.\label{sigma1}
\end{equation}
Moreover,
\begin{equation}
\tilde{F}(\xi)\leq\sigma_{1}\xi^{2} \ \ \mbox{and} \ \ \tilde f(\xi)\leq 2\sigma_1\xi, \ \qquad\mbox{for all}\ \xi\in\R\label{monotonicity1}
\end{equation}
Define now the  functional
\[
J:H^1_{V}(M)\to\R,\qquad J(u)=\int_{M}\alpha(x)\tilde{F}(u)\d,
\]
which is well defined, sequentially weakly continuous, Gâteaux differentiable
with derivative given by
\[
 J'(u)(v)=\int_{M}\alpha(x)\tilde{f}(u)v\d\qquad\mbox{for all}\ v\in H^1_{V}(M).
\]
Moreover, $J(0)=0$ and 

\begin{equation}
\sup_{u\in H^1_{V}(M)\setminus\{0\}}\frac{J(u)}{\|u\|_{V}^{2}}=\frac{\sigma_{1}}{\lambda_{\alpha}}.\label{sup}
\end{equation}
Indeed, from (\ref{monotonicity1}) immediately follows that
\[
\frac{J(u)}{\|u\|_{V}^{2}}\leq\frac{\sigma_{1}}{\lambda_{\alpha}}\;\mbox{for every }\ u\in H^1_{V}(M)\setminus\{0\}.
\]
Also, using the monotonicity assumption, for every $t>0,$ and for
every $x\in M$, such that $\varphi_{\alpha}(x)>0$
\[
\frac{\tilde{F}(t\varphi_{\alpha}(x))}{(t\varphi_{\alpha}(x))^{2}}\geq\frac{\tilde{F}(t\|\varphi_{\alpha}\|_{L^{\infty}(M)})}{t^{2}\|\varphi_{\alpha}\|_{L^{\infty}(M)}^{2}},
\]
thus
\begin{eqnarray*}
J(t\varphi_{\alpha}) & = & \int_{\{\varphi_{\alpha}>0\}}\alpha(x)\frac{\tilde{F}(t\varphi_{\alpha})}{(t\varphi_{\alpha})^{2}}(t\varphi_{\alpha})^{2}\d\geq\frac{\tilde{F}(t\|\varphi_{\alpha}\|_{L^{\infty}(M)})}{\|\varphi_{\alpha}\|_{L^{\infty}(M)}^{2}}\int_{M}\alpha(x)\varphi_{\alpha}^{2}\d\\
 & = & \frac{\tilde{F}(t\|\varphi_{\alpha}\|_{L^{\infty}(M)})}{\|\varphi_{\alpha}\|_{L^{\infty}(M)}^{2}}>0.
\end{eqnarray*}
Thus,
\[\frac{J(t\varphi_{\alpha})}{\|t\varphi_\alpha\|^2_V}=
\frac{J(t\varphi_{\alpha})}{t^2\lambda_\alpha}\geq
\frac{\tilde{F}(t\|\varphi_{\alpha}\|_{L^{\infty}(M)})}
{\|t\varphi_{\alpha}\|_{L^{\infty}(M)}^{2}} \frac{1}{\lambda_\alpha}.\]
Passing to the limit as $t\to0^{+}$, from (\ref{sigma1}), condition
(\ref{sup}) follows at once. Let us now apply Theorem \ref{D} with $X=H^1_V(M)$ and $J$ as above. Let
$r>0$ and denote by $\hat{u}$ the global maximum of $J_{|_{B_{x_0}(r)}}$.
We observe that $\hat{u}\neq0$ as $J(t\varphi_{\alpha})>0$ for every
$t$ small enough, thus $J(\hat u)>0$. If $\hat{u}\in$ int$(B_{x_0}(r))$, then, it turns
out to be a critical point of $J$, that is $J'(\hat{u})=0$ and (\ref{condition 1})
is satisfied. If $\|\hat{u}\|_V^{2}=r$, then, from the Lagrange multiplier
rule, there exists $\mu>0$ such that $J'(\hat{u})=\mu\hat{u}$, that
is, $\hat{u}$ is a solution of the equation

\[
-\Delta_{g}u+V(x)u=\frac{1}{\mu}\alpha(x)\tilde{f}(u), \ \mbox{in }M.
\]
Also,  by Theorem \ref{regularity}, $\hat{u}\in L^{\infty}(M)$ and
${\displaystyle \lim_{d_g(x_0,x)\to\infty}\hat{u}(x)=0}$. Condition (\ref{monotonicity2})
implies in addition that
\[
 J'(\hat{u})(\hat{u})-2J(\hat{u})=\int_{M}\alpha(x)[\tilde{f}(\hat{u})\hat{u}-2\tilde{F}(\hat{u})]\d\leq0.
\]
If the latter integral is zero, then, being $\alpha>0$, $\tilde{f}(\hat{u}(x))\hat{u}(x)-2\tilde{F}(\hat{u}(x))=0$
for all $x\in M$, which in turn implies that $\tilde{f}(\xi)\xi-2\tilde{F}(\xi)=0$
for all $\xi\in[0,\|\hat{u}\|_{L^{\infty}(M)}]$, that is, the function
$\xi\to\frac{\tilde{F}(\xi)}{\xi^{2}}$ is constant in the interval
$]0,\|\hat{u}\|_{L^{\infty}(M)}]$. In particular it would be constant
in a small neighborhood of zero which is in contradiction with the
assumption $(i)$. This means that (\ref{condition 1}) is fulfilled
and the thesis applies: there exists an interval $I\subseteq(0,+\infty)$
such that for every $\lambda\in I$ the functional
\[
u\to\frac{\|u\|_{V}^{2}}{2}-\lambda J(u)
\]
has a non-zero critical point $u_{\lambda}$ with $\ds\int_{M}(|\nabla u_{\lambda}|^{2}+V(x)u_{\lambda}^{2})\d<r.$
In particular, $u_{\lambda}$ turns out to be a nontrivial solution of the problem
\[
\left\{ \begin{array}{ll}
-\Delta_{g}u+V(x)u=\lambda\alpha(x)\tilde{f}(u), & \mbox{in }M\\
u\geq0, & \mbox{in }M\\
u\to0, & \mbox{as }d_{g}(x,x_{0})\to\infty.
\end{array}\right.\eqno{(\mathcal{\tilde{P}}_{\lambda})}
\]
From Remark \ref{I}, we know that $\ds I=\frac{1}{2}\left(\eta(r\delta_{r}),\lim_{s\to\beta_{r}}\eta(s)\right)$.
It is clear that
\[
\eta(r\delta_{r})=\sup_{y\in B_{r}}\frac{r-\|y\|_{V}^{2}}{r\delta_{r}-J(y)}\geq\frac{1}{\delta_{r}}
\]
and by the definition of $\delta_{r}$,
\[
\frac{r-\|y\|^{2}}{r\delta_{r}-J(y)}\leq\frac{r-\|y\|_{V}^{2}}{r\delta_{r}-\delta_{r}\|y\|_{V}^{2}}=\frac{1}{\delta_{r}}
\]
for every $y\in B_{r}$. Thus, recalling (\ref{sup}),
\[
\eta(r\delta_{r})=\frac{1}{\delta_{r}}=\frac{\lambda_{\alpha}}{\sigma_{1}}.
\]
Notice also that from Theorem \ref{regularity}, $u_{\lambda}\in L^{\infty}(M)$.
Let us prove that
\[
\lim_{\lambda\to\frac{\lambda_{\alpha}}{2\sigma_{1}}}\|u_{\lambda}\|_{L^\infty(M)}=0.
\]
Fix a sequence $\lambda_{n}\to\left(\frac{\lambda_{\alpha}}{2\sigma_{1}}\right)^{+}$.
Since $\|u_{\lambda_{n}}\|_{V}^{2}\leq r$, $\{u_{\lambda_{n}}\}$
admits a subsequence still denoted by $\{u_{\lambda_{n}}\}$ which
is weakly convergent to some $u_{0}\in B_{x_0}(r)$. Moreover, from the
compact embedding of $H^1_{V}(M)$ in $L^{2}(M)$, $\{u_{\lambda_{n}}\}$
converges (up to a subsequence) strongly to $u_{0}$ in $L^{2}(M)$.
Thus, being $u_{\lambda_{n}}$ a solution of $(\mathcal{P}_{\lambda_{n}})$,
\begin{equation}
\int_{M}(\nabla_{g}u_{\lambda_{n}}\nabla_{g}v+V(x)u_{\lambda_{n}}v)\d=\lambda_{n}\int_{M}\alpha(x)\tilde{f}(u_{\lambda_{n}})v\d\;\qquad \mbox{for all }\ v\in H^1_{V}(M),\label{weak}
\end{equation}
passing to the limit we obtain that $u_{0}$ is a solution of the
equation
\[
-\Delta_{g}u+V(x)u=\frac{\lambda_{\alpha}}{2\sigma_{1}}\alpha(x)\tilde{f}(u)\;\mbox{in }M.
\]
Assume $u_{0}\neq0$. Thus, testing (\ref{weak}) with $v=u_{\lambda_{n}}$,
\[
\|u_{\lambda_{n}}\|_{V}^{2}=\lambda_{n}\int_{M}\alpha(x)\tilde{f}(u_{\lambda_{n}})u_{\lambda_{n}}\d,
\]
and passing to the limit,
\begin{eqnarray*}
\|u_{0}\|_{V}^{2} & \leq & \liminf_{n\to\infty}\|u_{\lambda_{n}}\|_{V}^{2}=\frac{\lambda_{\alpha}}{2\sigma_{1}}\int_{M}\alpha(x)\tilde{f}(u_{0})u_{0}\d\\
 & < & \frac{\lambda_{\alpha}}{\sigma_{1}}\int_{M}\alpha(x)\tilde{F}(u_{0})\d\leq\lambda_{\alpha}\int_{M}\alpha(x)u_{0}^{2}\d\\
 & \leq & \|u_{0}\|_{V}^{2}.
\end{eqnarray*}
The above contradiction implies that $u_{0}=0$, and that ${\displaystyle \lim_{n\to\infty}\|u_{\lambda_{n}}\|_{V}=0}$.
Thus, in particular, because of the embedding into $L^{2^{\star}}(M)$,
we deduce that ${\displaystyle \lim_{n\to\infty}\|u_{\lambda_{n}}\|_{L^{2^{\star}}(M)}=0}$
and from Theorem \ref{regularity}, ${\displaystyle \lim_{n\to\infty}\|u_{\lambda_{n}}\|_{L^{\infty}(M)}=0}$.
Therefore,
\[
\lim_{\lambda\to\frac{\lambda_{\alpha}}{2\sigma_{1}}^{+}}\|u_{\lambda}\|_{L^{\infty}(M)}=0.
\]
This implies that there exists a number $\varepsilon_{r}>0$ such
that for every $\lambda\in\left(\frac{\lambda_{\alpha}}{2\sigma_{1}},\frac{\lambda_{\alpha}}{2\sigma_{1}}+\varepsilon_{r}\right)$,
$\|u_{\lambda}\|_{L^{\infty}(M)}\leq a$. Hence, $u_{\lambda}$ turns
out to be a solution of the original problem $(\mathcal{\mathscr{P}}_{\lambda})$
and the proof of this first case is concluded.

\medskip

Assume now $\sigma_1=+\infty$. The functional 
\[K:H^1_V(M)\to \R, \ K(u)=\int_M \alpha (x)F(u)dv_g.\]
is well defined and sequentially weakly continuous. 
Let $r>0$ and fix $\lambda\in I=\frac{1}{2}\left(0, \frac{1}{\lambda^*}\right)$ where 
\[\lambda^*= \inf_{\|y\|_V^2<r}\frac{\sup_{\|u\|_V^2\leq r}K(u)-K(y)}{r-\|y\|_V^2}\]
(with the convention $\frac{1}{\lambda^*}=+\infty$ if $\lambda^*=0$). 
Denote by $u_\lambda$ the global minimum of the restriction of the functional $\mathcal E$ to  $B_r$. Then, since 
\[\lim_{t\to 0}\frac{K(t \varphi_\alpha)}{\|t \varphi_\alpha\|_V^2}=+\infty,\]  it is easily seen that $\mathcal E(u_\lambda) <0$, therefore, $u_\lambda\neq 0$.
The choice of $\lambda$ implies, via easy computations, that $\|u_\lambda\|_V^2<r$. So, $u_\lambda$ is a critical point of $\mathcal E$, thus  a weak solution of $\Pl$.

\bigskip

\noindent $(ii)\Rightarrow(i)$. We follow the idea of \cite{Anello}. For the sake of completeness we give the details. Assume by contradiction that there
exist two positive constants $b,c$ such that
\[
\frac{F(\xi)}{\xi^{2}}=c\qquad\mbox{for all }\ \xi\in(0,b].
\]
Thus,
\begin{equation}
f(\xi)=2c\xi\qquad\mbox{for all }\ \xi\in[0,b].\label{equality}
\end{equation}
Let $\{r_{n}\}$ be a sequence of positive numbers such that $r_{n}\to0^{+}$.
Then, for every $n\in\mathbb{N}$ there exists an interval $I_{n}$
such that for every $\lambda\in I_{n}$, $(\mathscr{P}_{\lambda})$
has a solution $u_{\lambda,n}$ with $\|u_{\lambda,n}\|_V^{2}<r_{n}$.
Thus,
\[
\lim_{n}\sup_{\lambda\in I_{n}}\|u_{\lambda,n}\|_V=0.
\]
Since $f(\xi)\leq k(\xi+\xi^{q-1})$ for all $\xi\geq0$ (this follows
from the growth assumption (\ref{growth2}) and equality (\ref{equality})),
and being $u_{\lambda,n}$ a critical point of $\E$, from the continuous
embedding of $H^1_{V}(M)$ into $L^{2^{\star}}(M)$ and by Theorem \ref{regularity}
we obtain that
\[
\lim_{n}\sup_{\lambda\in I_{n}}\|u_{\lambda,n}\|_{\infty}=0.
\]
Let us fix $n_{0}$ big enough, such that ${\displaystyle \sup_{\lambda\in I_{n}}\|u_{\lambda,n}\|_{\infty}<b}$.
We deduce that for every $\lambda\in I_{n_{0}}$, $u_{\lambda,n_{0}}$
is a solution of the equation
\[
-\Delta_{g}u+V(x)u=2\lambda c\alpha(x)u,\mbox{ in }M,
\]
against the discreteness of the spectrum of the Schr\"{o}dinger operator
$-\Delta_{g}+V(x)$ established in Theorem \ref{C}.
\qed

\bigskip
\begin{rem}
	Notice that without the growth assumption (\ref{growth2}) the result holds true replacing the norm of the solutions $u_\lambda$ in the Sobolev space with the norm in $L^\infty(M).$
\end{rem}

We conclude the section with a corollary of the main result in the euclidean setting. We propose a more general set of assumption on $V$ which implies both the compactness of the embedding of $H^1_V(\R^N)$ into  and the discreteness of the spectrum of the Schr\"{o}dinger operator \cite{Benci-Fortunato-JMAA}.  Namely, let  $N\geq 3$, $\alpha:\R^N\to \R_+\setminus\{0\}$ be in $L^\infty(\R^N)\cap L^1(\R^N)$, $f:\R_+\to \R_+$ be a continuous function with $f(0)=0$ such that there exist two constants $k>0$ and $q\in(1, 2^\star)$ such that
\begin{equation*}
f(\xi)\leq k(1+\xi^{q-1}) \ \mbox{for all} \ \xi\geq 0.
\end{equation*}

Let also
$V:\R^N\to\R$ be in
$L_{\rm loc}^\infty(\R^N)$, such that  $\rm essinf_{\R^N}V\equiv V_0>0$ and
\[\int_{B(x)} \frac{1}{V(y)}dy \to 0 \qquad \mbox{as}  \ |x|\to \infty,\]
where $B(x)$ denotes the unit ball in $\R^N$ centered at $x$. In particular, if $V$ is a strictly positive ($\inf_{\R^N} V>0$), continuous and coercive function, the above conditions hold true.

\begin{corollary}
	Assume that for some $a>0$ the function $\xi\to \frac{F(\xi)}{\xi^2}$ is non-increasing in $(0,a]$. Then, the following conditions are equivalent:
	\begin{itemize}
		\item [(i)] for each $b>0$, the function $\xi\to \frac{F(\xi)}{\xi^2}$ is not constant in $(0,b]$;
		\item[(ii)] for each $r>0$, there exists an open interval $I\subseteq(0,+\infty)$ such that for every $\lambda\in I$, problem
		$$
		\left\{
		\begin{array}{ll}
		-\Delta u +V(x)u=
		\lambda \alpha(x) f(u), & \hbox{ in } \R^N \\
		u\geq 0, & \hbox{ in } \R^N \\
		u\to0, & \hbox{ as } |x|\to\infty
		\end{array}
		\right.
		$$
		 has a nontrivial solution $u_\lambda\in H^1(\R^N)$ satisfying $\ds \int_{\R^N} \left(|\nabla u_\lambda|^2 +V(x)u_\lambda^2\right)dx<r$.
	\end{itemize}
\end{corollary}

\noindent \textbf{Acknowledgments.} This work was initiated when Cs.
Farkas visited the Department of Mathematics of the University of
Catania, Italy. F.Faraci is member of the Gruppo Nazionale per l'Analisi Matematica, la Probabilit\`{a}
e le loro Applicazioni (GNAMPA) of the Istituto Nazionale di Alta Matematica (INdAM).

\def\cprime{$'$} \def\cprime{$'$}

\end{document}